\documentclass[11pt]{article}

\usepackage{amsmath,bbm,amssymb,amsthm,epsfig}
\usepackage{graphicx}
\theoremstyle{definition}
\newtheorem{thm}{Theorem}[section]
\newtheorem{lem}{Lemma}[section]
\newtheorem{cor}{Corollary}[section]

\newtheorem*{remark}{Remark}

\newcommand{\R}{\mathbbm{R}}                 
\newcommand{\N}{\mathbbm{N}} 
\newcommand{\Z}{\mathbbm{Z}}                 

\newcommand{\eps}{\epsilon}

\newcommand{\bigo}[1]{\mathcal{O}\left(#1\right)}

\renewcommand{\S}[3]{\displaystyle{\sum_{#1}^{#2}{#3}}}

\begin{document}


\title{A Local Limit Theorem and Loss of Rotational Symmetry of Planar Symmetric Simple Random Walk}
\author{Christian Bene\v{s}}
\date{}
\maketitle

\begin{abstract}
We 
derive a 
local limit theorem 
for normal, moderate, and large deviations
for symmetric simple random walk on the square lattice in dimensions one and two that is an improvement of existing results for points that are particularly distant from the walk's starting point. More specifically, we give explicit asymptotic expressions in terms of $n$ and $x$, where $x$ is thought of as dependent on $n$, in dimensions one and two for $P(S_n=x)$, the probability that symmetric simple random walk $S$ started at the origin is at some point $x$ at time $n$, that are valid for all $x$. 
We also show that the behavior of planar symmetric simple random walk differs radically from that of planar standard Brownian motion outside of the disk of radius $n^{3/4}$, where the random walk ceases to be approximately rotationally symmetric. Indeed, if $n^{3/4}=o(|S_n|)$, $S_n$ is more likely to be found along the coordinate axes. In this paper, we give a description of how the transition from approximate rotational symmetry to complete concentration of $S$ along the coordinate axes occurs.
\end{abstract}


\section{Introduction}\label{SectionIntro}

Symmetric simple random walk $S$ in dimension $d\in\N$ is defined by
 $S(0)=0$ and for $n\in\N$, by $S_n = \sum_{k=1}^{n}X_k$,
where $\{X_k\}_{k\in\N}$ are independent random vectors satisfying
$P(X_k = \pm e_i) = \frac{1}{2d}, i=1,\ldots, d$
and $\{e_i\}_{i\in\{1,\ldots, d\}}$ is the standard orthonormal basis of $\R^d$. 

Donsker's invariance principle (see \cite{donsker}) tells us that a large class of rescaled random walks converge in distribution 
to standard Brownian motion. This, together with a number of strong approximation results (see \cite{revesz} for an extensive survey of the topic), suggests that Brownian motion and random walk have similar behavior ``at large scales" in all dimensions. 
This is often a good way of thinking about these processes and one can show, in particular with the help of coupling arguments, (see \cite{puckette1}, \cite{puckette2}, and \cite{holes} for examples) that random walk and Brownian motion share many properties. However, as will be made clear in this paper, even at macroscopic scales, there are fundamental differences between the two processes.

While the transition density for $d$-dimensional Brownian motion has the well-known exact expression  
$p_t(x)=(2\pi t)^{-d/2}e^{-\|x\|_2^2/2t}$ (see for instance \cite{MortersPeres}), where $x\in\R^d, t\in\R$, and, here and throughout this paper, $\|\cdot \|_p$ denotes $L^p$ norm, the expression for the probability 
\begin{equation}\label{mainprob}
P(S_n=x)
\end{equation}
that $d$-dimensional symmetric simple random walk $S$ started at the origin is at a given location $x=(x_1, \ldots , x_d) \in\Z^d$ at time $n\in\N$ cannot immediately be expressed as a convenient function of $x$ and $n$. The multidimensional central limit theorem (see for instance \cite{jacodprotter}) suggests that if $\sum_{i=1}^d x_i+n$ is even, this probability is more or less the same as that of Brownian motion being in a ball centered at $x\in\R^d$, of volume 2 (the factor 2 accounts for the fact that simple random walk has period 2) at time $n/d$, which is approximately $2(d/2\pi n)^{d/2}\exp\{-d\|x\|_2^2/2n\}$. 
As the reader can see below in the cases $d=1, 2$, this argument is correct for points in and somewhat beyond the disk of radius $\sqrt{n}$, that is, the typical range of $S_n$, but in fact also for points up to a distance of about $n^{3/4}$. However, it fails for points beyond that distance and thus for a large majority of points that are attainable by the walk at time $n$. 

Estimating $P(S_n=x)$ for random walks under a wide variety of assumptions has been the object of study of numerous authors since the first half of the twentieth century. In his seminal paper (\cite{polya}), P\'{o}lya showed that for $d$-dimensional symmetric simple random walk $S$ if $x=(x_1, \ldots , x_d) $ is fixed and $\sum_{i=1}^d x_i+n$ is even, 
\begin{equation}\label{polyaeq}
\lim_{n\to\infty} n^{d/2}P(S_n=x)=2\left(\frac{d}{2\pi}\right)^{d/2}.
\end{equation}
See also \cite{erdostaylor}, \cite{greenbook}, and \cite{lawlerlimic} for refinements of this result with error terms. The equality in \eqref{polyaeq}, together with the fact that $P(\|S_n\|_2\geq r\sqrt{n})$ decays exponentially with $r$ (see Corollary 3.1 in \cite{benesnotes} for the one-dimensional case; the $d$-dimensional case then follows easily by using (2.1) in \cite{lawlerlimic}) suggests that the distribution of the random walk at time $n$ is roughly uniform in a disk of radius of order about $\sqrt{n}$ while the likelihood that the walk is outside of a disk of radius of any order larger than $\sqrt{n}$ is negligible.


The behavior of a random walk at time $n$ up to distances of order greater than $\sqrt{n}$ belongs to the realms of moderate and large deviations. When trying to understand the large deviations behavior of a random walk, one is interested in the probability of finding it at distance of order $n$ at time $n$, whereas the moderate deviations behavior of a walk is its behavior at time $n$ at a distance of order strictly greater than $\sqrt{n}$ but strictly smaller than $n$. While results on moderate and large deviations are of a slightly different nature than local limit theorems, there is a strong connection between the two through the large deviations rate function which we define below. Therefore, we briefly mention here some of the basic definitions and facts from large deviation theory that are relevant to us and discuss the connection with local limit theorems in more detail in the Appendix.


Large deviations theory has become an important and vast area of probability theory after Varadhan introduced abstract notions of a general theory in \cite{varadhan}, though some important ideas go back to Boltzmann (\cite{boltzmann}) and Cram\'{e}r (\cite{cramer}). We will only define here the main notions in the specific context of this paper. For a much more general treatment of  the theory, the reader may consult, for instance, \cite{dembozeitouni}. 

We let $\mathcal{X}$ denote the convex hull of the points of the increment distribution of the random walk. In particular, for symmetric simple random walk in $\Z^d, \mathcal{X}=\{x\in\R^d:\|x\|_1\leq 1\}$
. A rate function is a mapping $I:\mathcal{X}\to \R_+$ such that for every $a\in\R_+$, the level set $\{x:I(x)\leq a\}$ is closed. If, moreover, the level sets are compact, we say the rate function is good. Clearly, a rate function for a random walk with bounded increments is always good.

Let $(b_n)$ be a positive sequence such that $b_n/n^{1/2}\to\infty$ as $n\to\infty$ and let $(a_n)$ be a sequence. We say that the sequence $(S_n/b_n)$ satisfies a large deviations principle with rate function $I$ and speed $(a_n)$ if for all $\Gamma\in\mathcal{B}(\mathcal{X})$, that is, for all Borel sets in $\mathcal{X}$, 
$$\!-\!\inf_{\alpha\in\Gamma^{o}}\!I(\alpha)\leq \liminf_{n\to\infty}\frac{1}{a_n}\!\log P\left(\frac{S_n}{b_n}\in \Gamma\right)\leq \limsup_{n\to\infty}\frac{1}{a_n}\!\log P\left(\frac{S_n}{b_n}\in \Gamma\right)\leq \!-\!\inf_{\alpha\in\overline{\Gamma}}\!I(\alpha).$$
In the cases where 
\begin{equation}\label{moddev}
b_n/n\to 0 \text{ as } n\to\infty,
\end{equation}
one often refers to ``moderate deviations" in the literature. For more details, see \cite{dembozeitouni} and \cite{eichelsbacherlowe}. Throughout this paper, when we talk about a random walk's large deviations rate function, we mean the rate function for the sequence $(S_n/n)$.


Moderate and large deviations principles are satisfied by random walks under a variety of assumptions on the increment distribution. The  moderate deviations behavior of a large class of random walks is essentially Gaussian (see Theorem 3.7.1 in \cite{dembozeitouni}). In particular, for symmetric simple random walk in $\Z^d$, under the assumption \eqref{moddev}, that is, in the moderate deviations regime, $I(\alpha)=-\frac{d}{2}\|x\|_2^2$. We will see below that, in fact, this moderate deviations rate function misses much of the fine detail of the decay of the probability in \eqref{mainprob}. The large deviations rate function for random walks is known to be the Legendre-Fenchel transform of the logarithmic moment generating function (see \cite{dembozeitouni} and also the Appendix of this paper for a brief argument for why this should be the case):
$$\Lambda(\alpha)=\sup_{\lambda\in\R^d}\{\langle \lambda, \alpha\rangle-\log M(\lambda)\}
.$$

There is a long history of results leading to very precise local limit theorems for large classes of random walks (see \cite{borovkovbis} for what is probably the most extensive monograph on the topic, in particular Chapter 6 for the case discussed in this paper). For the purpose of comparison with the present work, we include in the appendix a brief description of the results from some of the papers that made fundamental contributions to the problem. We point out that the results we present in this paper for the particular case of the simple random walk are not contained in any of them.

Throughout this paper, the notation $f(n) =\bigo{g(n)}$ will mean that there is a universal constant $C$ such that $|f(n)|\leq C|g(n)|$ for all $n$. If the constant $C$ depends on some other quantity, this will be mentioned or made explicit inside the $\bigo{\cdot}$. We will write $f(n) \sim g(n)$ whenever $\lim_{n\to\infty}f(n)/g(n)=1$
. For notational simplicity, multiplicative constants will generally all be denoted by $C$, though they may be different from one line to the next.

This paper is organized as follows: In Section \ref{Results}, we state the one-dimensional and two-dimensional local limit theorems and the corollary to our planar result which shows where planar symmetric simple random walk is rotationally symmetric and where it isn't. In Sections \ref{1d} and \ref{2d}, we provide the proofs in the, respectively, one-dimensional and two-dimensional case. Finally, the Appendix presents some related results and shows explicitly that the exponential decay we obtain for \eqref{mainprob} is indeed dictated by the simple random walk large deviations rate function, as expected from the discussion in the first part of the Appendix.

In order to be able to get asymptotics for \eqref{mainprob} for the entire range of simple random walk, the approach based on the Fourier transform that we present in the Appendix seems unlikely to be successful. In this paper, we take a more direct approach and use Stirling's formula to approximate the binomial coefficients that appear in the expressions of \eqref{mainprob}. This is relatively straightforward in dimension one but requires considerably more work in dimension two where approaching the question from two different angles leads to two different expressions. One of these expressions gives asymptotic values for \eqref{mainprob} in the entire range of the random walk. The other doesn't, but has the advantage 
of being obtained by using a method that can be extended to dimensions greater than two. Both expressions can be used to show that planar simple random walk is essentially rotationally invariant (in a sense to be made precise below) in the disk of radius $n^{3/4}$ and ceases to be so at greater distances. 

\section{Statement of Results}\label{Results}

As mentioned in the introduction and discussed further in the appendix, there are a number of abstract local limit theorems available in the literature that are valid well beyond the typical range of a random walk. There are also a number of local limit theorems (see, for instance, \cite{spitzer}, \cite{greenbook}, \cite{woess}, \cite{uchiyama}, and \cite{lawlerlimic}) that give explicit asymptotic expressions for the probability in \eqref{mainprob}. The aim of this paper is to reconcile the advantages of the two types of results in the symmetric simple random walk case by providing explicit asymptotic expressions for $P(S_n=x)$ in dimensions 1 and 2 that are valid for 
all $x$
. 

Theorem \ref{LCLTthm2d} below can be used to show that in dimension 2, an interesting phenomenon arises: For points $x\in\Z^2$ with $|x|>> n^{3/4}$, some points along a same discrete circle are much more likely to be hit by $S$ than others. More precisely, for any $r=r(n)\in\R$ with $n^{3/4}=o(r)$, there exist points $x_1, x_2\in\Z^2$ with $|x_1|, |x_2|\in [r,r+2]$ such that $\lim_{n\to\infty} P(S_n=x_1)/P(S_n=x_2)=0$. In other words, while planar symmetric simple random walk 
is essentially rotationally symmetric (we make this precise below) in the disk of radius $n^{3/4}$, it loses this symmetry outside of that disk. This is the content of Corollary \ref{coro} below.




The first result of this paper is a one-dimensional local  limit theorem which gives the asymptotic behavior of $P(S_n=x)$ for all $x$.
It extends the result of \cite{khinchine}, valid for $x=o(n)$, to all $x$.
\begin{thm}\label{LCLTthm1d}
Suppose $S$ is one-dimensional symmetric simple random walk and $n\in\N$.
Suppose
\begin{itemize}

\item $c\in\N$ is a constant with $c<n$;

\item $\ell_1, \ell_3$ are 
functions of $n$ such that $\ell_1(n), \ell_3(n)=o(n), \ell_3(n)\to\infty$ as $n\to\infty$;

\item for $0<a<1, \ell^a_2$ is a function of $n$ such that 
$\ell^a_2(n)=o(n)$
.
\end{itemize}
For $x\in\Z$, suppose, moreover, that the constant $c$ and functions $\ell_1, \ell^a_2, \ell_3$ are chosen below so that $-n\leq x\leq n$ and $x+n$ is an even integer. Then
\begin{equation}\label{1dimprob}
P(S_n=x) =   \exp\left(\phi(n,x)\right) f_x(n),
\end{equation}
where
\begin{equation}\label{1dimdetails}
f_x(n)=\left\{
\begin{array}{ll} 
\sqrt{\frac{2}{\pi n}}\left(1+\bigo{\frac{\ell_1(n)^2}{n^2}+\frac{1}{n}}\right), & |x|=\ell_1(n),\\
\sqrt{\frac{2}{\pi n}}\frac{1}{\sqrt{1-a^2}}\left(1\!+\!\bigo{\frac{\ell_2^a(n)}{n(1-a^2)}}\right), & |x|= an\!+\!\ell_2^a(n), 0<a<1,\\
\frac{1}{\sqrt{\pi \ell_3(n)}}\left(1+\bigo{\frac{\ell_3(n)}{n}}+\bigo{\frac{1}{\ell_3(n)}}\right), & |x| = n-\ell_3(n),\\
\frac{de^{-1/6c}}{\sqrt{\pi c}}\left(1+\bigo{\frac{c}{n}}\right), & |x| = n-c,\\
1, & |x|=n,
\end{array}
\right.
\end{equation}
\begin{equation}\label{phinx}
\phi(n,x) =
 - \S{\ell=1}{\infty}{\frac{1}{2\ell(2\ell-1)}\frac{x^{2\ell}}{n^{2\ell-1}}},
 \end{equation}
 $d\in [\frac{e^{11/12}}{\sqrt{2\pi}},\frac{2\pi}{e^{11/6}})$, and the constants in the $\bigo{\cdot}$ terms are universal. 
In particular, for every $n\in\N$, every $N\geq 2, s<\frac{2N-1}{2N}$, and $|x|\leq n^s$, 
\begin{equation}\label{1das}
P(S_n=x) = \sqrt{\frac{2}{\pi n}}\exp\left(-\S{l=1}{N-1}{\frac{1}{2\ell(2\ell-1)}\frac{x^{2\ell}}{n^{2\ell-1}}}\right)
    \left(1+\bigo{
    \frac{x^{2N}}{n^{2N-1}}+\frac{x^2}{n^2}+
    \frac{1}{n}}\right).
\end{equation}
\end{thm}

\begin{remark} While the penultimate line in \eqref{1dimdetails} does not provide exact asymptotics, since $d$ lies in an interval (albeit a short one, as $\frac{e^{11/12}}{\sqrt{2\pi}}\approx 0.9977$ and $\frac{2\pi}{e^{11/6}}\approx1.0046$), we of course have the alternative expression 
$$P(S_n=n-c)={n\choose \frac{2n-c}{2}}\left(\frac{1}{2}\right)^n$$
which can be computed exactly for any constant $c$.

\end{remark}

\begin{remark} Note that the transitions between the regimes in Theorem \ref{LCLTthm1d} work as they should. Indeed, when setting $a=0$ in the second line of \eqref{1dimdetails}, one obtains the main order term of the first line and when setting $a=1-\ell_3(n)/n$, so that $an=n-\ell_3(n)$, one obtains the main order term of the third line.
\end{remark}

\begin{remark} As mentioned in the Introduction, one can verify that the sum in \eqref{phinx} is equal to $n\Lambda(x/n)$ where $\Lambda$ is the large deviations rate function for symmetric simple random walk. We do not do this verification here, but outline the main steps of the considerably more complicated two-dimensional case in the Appendix.
\end{remark}
 
Note that in Theorem \ref{LCLTthm1d}, one should, as mentioned before, think of $x$ as being dependent on $n$ and going to infinity, since the error terms all vanish only as $n\to\infty$, so that for fixed $x$ with $x+n$ even, the limiting probability is just $\sqrt{2/\pi n}$ as shown by P\'{o}lya in \cite{polya}.

One consequence of Theorem \ref{LCLTthm1d} is that for one-dimensional symmetric simple random walk, the exponent in the probability $P(S_n=x)$ is the same as for the corresponding Brownian motion probability 
of being in a ball of volume 2 (recall that this is to take into account the periodicity of the random walk) around $x$, namely $-|x|^2/2$, as long as $|x| = \bigo{n^{3/4}}$. However as soon as $|x| >>n^{3/4}$, the random walk probability has an additional exponential term which makes the random walk probability smaller than the corresponding Brownian motion probability. This should of course not be surprising, since for $x$ with $|x|>n, P(S_n=x)=0$, while $P(B(n)=x)>0$. Theorem \ref{LCLTthm1d} describes precisely the transition from the exponent $-|x|^2/2$ to a probability of 0 as the magnitude of $x$ goes from $n^{3/4}$ to $n$. 

This difference in behavior between symmetric simple random walk and standard Brownian motion also exists in the plane, though another interesting phenomenon arises in that case: If $n^{3/4}=o(\|(x,y)\|_2)$, the asymptotic behavior of $P(S_n=(x,y))$ depends on the location of $(x,y)$, not just on $\|(x,y)\|_2$:

\begin{thm}\label{LCLTthm2d}
Let $S$ be a planar symmetric simple random walk and $n\in\N$. Let 
\begin{itemize}
\setlength\itemsep{0.001em}
\item $\ell_1, \ell_1'$ be integer-valued functions of $n$ such that $\ell_1(n), \ell_1'(n)=o(n)$;
\item $0<a<1/2, 0<a_x < a_y<1$ with $a_x+a_y\leq 1, 0<b<1$ and $\ell^a_2$, $\tilde{\ell}^a_2$, $\ell^b_2$, $\ell^{a_x}_2$, $\ell^{a_y}_2$ be functions of $n$ such that 
$\ell^a_2(n)$, $\tilde{\ell}^a_2(n)$, $\ell^b_2(n)$, $\ell^{a_x}_2(n)$, $\ell^{a_y}_2(n)=o(n)$;
\item $\ell_3, \ell_3'$ be functions of $n$ such that $\ell_3(n), \ell_3'(n)=o(n), \ell_3(n), \ell_3'(n)\to\infty$ as $n\to\infty$;
\item $c, c_x, c_y
>0$ be constants;

\item $d$ be as in Theorem \ref{LCLTthm1d}. 
\end{itemize}
For $(x,y)\in\Z^2$, suppose, moreover, that the constants $c, c_x, c_y$ and functions $\ell_1, \ell_1', \ell^a_2, \tilde{\ell}^a_2$, $\ell^b_2$, $\ell^{a_x}_2$, $\ell^{a_y}_2, \ell_3, \ell_3'$ are chosen below so that $\|(x,y)\|_1\leq n$ and $x+y+n$ is an even integer. Then
\begin{equation}\label{general}
P(S_n=(x,y)) =  f_{x,y}(n) \cdot \exp\left\{-\sum_{\ell =1}^{\infty} \frac{1}{2\ell(2\ell-1)}\frac{(x+y)^{2\ell}+(y-x)^{2\ell}}{n^{2\ell -1}}\right\},
\end{equation}
where, if $0\leq x\leq y$, with the notation $a_-=a_y-a_x, a_+=a_x+a_y$, 

\begin{equation}\label{2ddetails}
f_{x,y}(n) \!=\!   \left\{
\begin{array}{ll}
\frac{2}{\pi n}\left(1 + \bigo{\frac{n+x^2+y^2}{n^2}}\!\right), & \|(x,y)\|_1=o(n),\\ [.7cm]
\frac{2}{\pi n\sqrt{1\!-\!4a^2}}\!\left(\!1 \!\!+ \!\!\mathcal{O}\!\left(\!\frac{n|\ell_2^{a}(n)+\tilde{\ell}_2^{a}(n)|+n+(\tilde{\ell}_2^{a}(n)-\ell_2^{a}(n))^2}{n^2}\!\right)\!\right)\!,
& \begin{array}{ll}
\hspace{-0.1pc}\!\!x\!=\!an\!+\!\ell_2^{a}(n),\\ 
\hspace{-0.1pc}\!\!y\!=\!an\!+\!\tilde{\ell}_2^{a}(n), 
\end{array}
\\[.7cm]
\frac{2}{\pi n(1-b^2)}\left(1 + \bigo{\frac{|\ell_1(n)|+|\ell_2^{a}(n)|}{n}}\right), & 
\begin{array}{ll}
\hspace{-0.1pc}\!\!x=\ell_1(n), \\
\hspace{-0.1pc}\!\! y=bn+\ell_2^{b}(n),
\end{array} \\
[.7cm]
\frac{2}{\pi n\sqrt{1\!-\!2(a_x^2+a_y^2)\!+\!(a_-a_+)^2}}\!\left(\!1 \!\!+\!\! \bigo{\!\frac{|\ell_2^{a_x}(n)|+|\ell_2^{a_y}(n)|}{n}\!}\! \right), & \begin{array}{ll}
\hspace{-0.1pc}\!\!x\!=\!a_xn\!+\!\ell_2^{a_x}(n),\\ 
\hspace{-0.1pc}\!\!y\!=\!a_yn\!+\!\ell_2^{a_y}(n), \\
\hspace{-0.1pc}\!\!a_x+a_y<1,
\end{array}
\\[.7cm]
\frac{\sqrt{2}}{\pi \sqrt{n\ell_3(n)}}\left(1 \!\!+\!\! \bigo{\frac{(\ell_1'(n)-\ell_1(n))^2}{n^2}\!+\!\frac{\ell_3(n)}{n}\!+\!\frac{1}{\ell_3(n)}}\right), &  \begin{array}{ll}
\hspace{-0.1pc}\!\!x\!=\!n/2\!+\!\ell_1(n),\\
\hspace{-0.1pc}\!\! y\!=\!n/2\!+\!\ell_1'(n),\\
\hspace{-0.1pc}\!\!x+y=n-\ell_3(n), 
\end{array}
\\[.7cm]
\frac{\sqrt{2}}{\pi \sqrt{n(1\!-\!a_-^2)\ell_3(n)}}\left(\!1 \!\!+\!\! \bigo{\frac{|\ell_2^{a_y}(n)|+|\ell_2^{a_x}(n)|}{n}\!+\!\frac{1}{\ell_3(n)}}\right)\!, &  \begin{array}{ll}
\hspace{-0.1pc}\!\!x\!=\!a_xn\!+\!\ell_2^{a_x}(n), \\
\hspace{-0.1pc}\!\!y\!=\!a_yn\!+\!\ell_2^{a_y}(n),\\
\hspace{-0.1pc}\!\!x+y=n-\ell_3(n),
\end{array}
\\[.7cm]
\frac{1}{\pi \sqrt{\ell_3^2(n)-\ell_1^2(n)}}\left(1 + \bigo{\frac{\ell_3(n)}{n}+\frac{1}{|\ell_3'(n)}|}\right), & 
\begin{array}{ll}
\hspace{-0.1pc}\!\! x=\ell_1(n),\\
\hspace{-0.1pc}\!\! y=n-\ell_3(n), \\
\hspace{-0.1pc}\!\! \ell_3(n)\!\!-\!\ell_1(n)\!=\!\ell_3'(n),
\end{array}
\\[.7cm]
\frac{d\sqrt{2}e^{-1/6c}}{\pi \sqrt{nc}}\left(1  + \bigo{\frac{(\ell_1'(n)-\ell_1(n))^2+n}{n^2}}\right), &  \begin{array}{ll}
\hspace{-0.1pc}\!\!x\!=\!n/2\!+\!\ell_1(n),\\
\hspace{-0.1pc}\!\! y\!=\!n/2\!+\!\ell_1'(n),\\
\hspace{-0.1pc}\!\!x+y=n-c,
\end{array}
\\[.7cm]
\frac{d\sqrt{2}e^{-1/6c}}{\pi \sqrt{nc(1-a_-^2)}}\left(1  + \bigo{\frac{1+|\ell_2^{a_y}(n)-\ell_2^{a_x}(n)|}{n}}\right), &  \begin{array}{ll}
\hspace{-0.1pc}\!\!x\!=\!a_xn\!+\!\ell_2^{a_x}(n),\\
\hspace{-0.1pc}\!\! y\!=\!a_yn\!+\!\ell_2^{a_y}(n),\\
\hspace{-0.1pc}\!\!x+y=n-c,
\end{array}
\\[.7cm]
\frac{de^{-1/c}}{\pi \sqrt{c(\ell_1(n)+\ell_3(n))}}\left(1 + \bigo{\frac{\ell_1(n)}{n}}+ \bigo{\frac{1}{\ell_1(n)}}\right), &  \begin{array}{ll}
\hspace{-0.1pc}\!\!x=\ell_1(n),\\ 
\hspace{-0.1pc}\!\!y=n-\ell_3(n), \\
\hspace{-0.1pc}\!\!\ell_3(n)-\ell_1(n)=c,
\end{array}
\\[.7cm]
\frac{d^2e^{-c_y/3(c_y^2-c_x^2)}}{\pi \sqrt{c_y^2-c_x^2}}\left(1 + \bigo{\frac{1}{n}}\right), & x=c_x, y=n\!-\!c_y,\\[.5cm]
\sqrt{\frac{2}{\pi n}}\left(1 + \bigo{\frac{\ell_3(n)^2+n}{n^2}}\right), &  \begin{array}{ll}
\hspace{-0.1pc}\!\!x\!=\!n/2\!-\!\ell_3(n),\\
\hspace{-0.1pc}\!\! y\!=\!n/2\!+\!\ell_3(n),
\end{array}
\\[.7cm]
\sqrt{\frac{2}{\pi n(1-a_-^2)}}\left(1 + \bigo{\frac{\ell_2^{a_x}(n)}{n}}\right), &  \begin{array}{ll}
\hspace{-0.1pc}\!\!x\!=\!a_xn\!+\!\ell_2^{a_x}(n),\\
\hspace{-0.1pc}\!\! y\!=\!a_yn\!+\!\ell_2^{a_y}(n),\\
\hspace{-0.1pc}\!\! x+y=n,
\end{array}
\\[.7cm]
\frac{1}{\sqrt{2\pi \ell_3(n)}}\left(1 \!\!+\!\! \bigo{\frac{\ell_3(n)}{n}} + \bigo{\frac{1}{\ell_3(n)}} \right), & \begin{array}{ll}
\hspace{-0.1pc}\!\!x\!=\!\ell_3(n),\\ 
\hspace{-0.1pc}\!\!y\!=\!n\!-\!\ell_3(n), 
\end{array}
\\[.7cm]
\frac{de^{-1/12c}}{\sqrt{2\pi c}} \left(1 + \bigo{\frac{1}{n}}\right), & x=c, y=n-c,\\[.5cm]
1, & x=0, y=n,
\end{array}
\right.
\end{equation}

 \begin{figure}[htb!]\label{lpdisks}
\centering%
\includegraphics[scale=0.9]{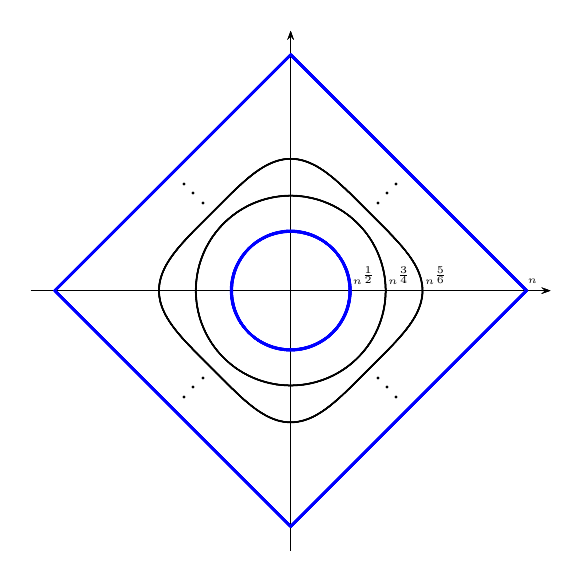}
\caption{
Equation \eqref{general} shows that  
a new term appears in the sum in the exponent for $P(S_n=(x,y))$ every time $(x,y)$ reaches a new $L^{2\ell}$ ball of radius $n^{1-1/2\ell}, \ell\geq 1$, with axes $x+y$ and $x-y$; Corollary \ref{coro} shows that rotational invariance does not extend beyond the disk of radius $n^{3/4}$.}
\end{figure}

\noindent where the $\bigo{\cdot}$ terms may depend on the constants $a, a_x, a_y, b, c, c_x$, and $c_y$, and for general $x,y, f_{x,y}(n)$ is obtained from the values above by the axial symmetries of the walk.

Moreover, if $x^2+y^2
\leq n^2/\log^4(n)$ 
and $|y^2-x^2|\leq n^{3/2}/\log(n)$, we have
\begin{equation}\label{regular}
P(S_n=(x,y))  =\frac{2}{\pi n} e^{-s_0} \left(1\!
 +\bigo{\frac{y^2-x^2}{n^2}}+\bigo{\frac{(y^2-x^2)^2}{n^3}}\right)
\end{equation}
where 
$$s_0=\sum_{\ell \geq 1} \frac{1}{4\ell(2\ell\! -\!1)}\frac{(2x)^{2\ell}+(2y)^{2\ell}}{n^{2\ell -1}}.$$

In particular, for all $n\in\N, N\geq 2, a<(2N-1)/2N$, and $(x,y)$ such that $\sqrt{x^2+y^2}\leq n^a$ and $|y^2-x^2|\leq n^{3/2}/\log(n)$, we have 
\begin{align} \notag
P(S_n=(x,y)) = & \frac{2}{\pi n} \exp\left\{-\sum_{\ell = 1}^{N-1} \frac{1}{4\ell(2\ell\! -\!1)}\frac{(2x)^{2\ell}+(2y)^{2\ell}}{n^{2\ell -1}}\right\}\\ \notag
& \hspace{2pc} \cdot \left(1\!
 +\!\bigo{\frac{y^2-x^2}{n^2}}\!+\!\bigo{\frac{(y^2-x^2)^2}{n^3}} \!+\!\bigo{\frac{(x^2+y^2)^N}{N^{2N-1}}} \right).
\end{align}
\end{thm}

\begin{remark}
The different cases listed in \eqref{2ddetails} are ordered so that the variable $x+y$ goes in increasing order as in the cases of \eqref{1dimdetails}, and, within cases where $x+y$ is of the same order, so that the variable $y-x$ goes in increasing order as in \eqref{1dimdetails}. 
\end{remark}

\begin{figure}[htb!]
\centering%
\includegraphics[scale=0.5]{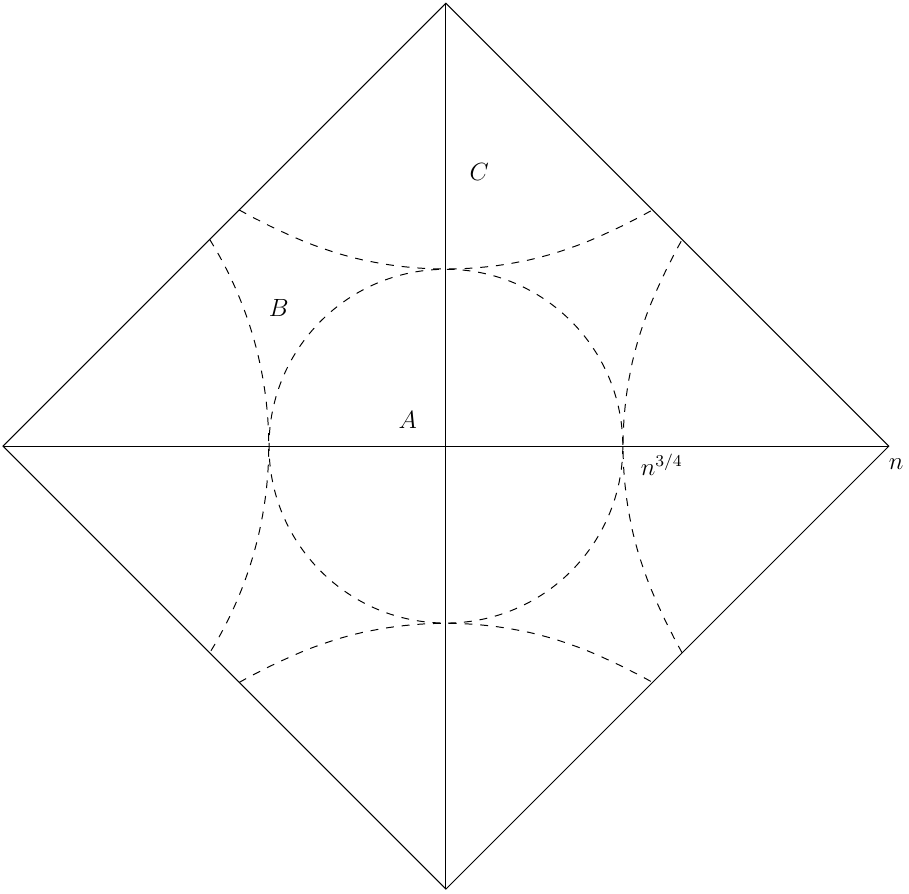}
\caption{
The picture in two dimensions: The square $\{(x,y)\in\Z^2:|x+y|\leq n\}$ is the set of attainable points by $S_n$. Regions $A, B$, and $C$ correspond to the different regimes; In region $A, P(S_n=(x,y))$ is as for Brownian motion; in region $B$, the probability is logarithmically asymptotic to $-\sum_{\ell \geq 1} \frac{1}{4\ell(2\ell\! -\!1)}\frac{(2x)^{2\ell}+(2y)^{2\ell}}{n^{2\ell -1}}$. In region $C$ there is an additional exponential correction term.}
\end{figure}

\begin{remark}
As in the one-dimensional case, the lines in \eqref{2ddetails} that contain the unknown $d$ can be rewritten in terms of binomial coefficients since, as can be seen in the proof of Theorem \ref{LCLTthm2d}, \eqref{general} is obtained by considering the probabilities that independent one-dimensional simple random walks are at $x+y$ and $x-y$, respectively.
\end{remark}

\begin{remark}
While this may not be obvious algebraically, the two expressions in \eqref{general} and \eqref{regular} are of course equivalent when $x^2+y^2
\leq n^2/\log^4(n)$ 
and $|y^2-x^2|\leq n^{3/2}/\log(n)$. While the first expression applies to a larger range of points, the method that leads to it cannot be extended to higher dimensions. The method that leads to \eqref{regular}, however, carries the advantage of being extensible to higher dimensions.
\end{remark}

\begin{remark} The exponent in \eqref{general} is $-n\Lambda(x/n)$ where $\Lambda$ is the large deviations rate function for symmetric simple random walk. See the Appendix for more details.
\end{remark}

We will say that planar simple random walk is approximately rotationally symmetric in the disk of radius $r=r(n)$ if  for all $s\leq r$ and all points $x,x'\in\Z^2$ with $s-3\leq \|x\|_2, \|x'\|_2\leq s$ and such that $\|x\|_1+n$ and $\|x'\|_1+n$ are both even, we have $P(S_n=x)\sim P(S_n=x')$ as $n\to\infty$. Donsker's invariance principle says that if $S$ is planar simple random walk interpolated linearly between integer times and for $n\in\N, 0\leq t\leq 1$, one defines $\tilde{S}_n(t)=\frac{1}{\sqrt{n}}S(2nt)$, then the sequence $\tilde{S}_n$ converges weakly to planar standard Brownian motion $\{B(t):0\leq t\leq 1\}$ on $\mathcal{C}[0,1]$. Corollary \ref{coro} shows that this principle misses some of the subtle differences between $B$ and $S$. Indeed, on the rare events where $\tilde{S}$ goes beyond the circle 
of radius $n^{1/4}$, the distribution of the paths of $\tilde{S}$ differs radically from that of its scaling limit, since the latter is rotationally symmetric, while the former is much more likely to be along the coordinate axes than on the diagonals of $(1/\sqrt{n})\Z^2$:

\begin{cor}\label{coro} Planar symmetric simple random walk is approximately rotationally symmetric in the disk of radius $r=r(n)$ when $r=o(n^{3/4})$. However, for every $n^{3/4}\leq r\leq n$, there are points $x, x'\in\Z^2$ with $\|x\|_2, \|x'\|_2\in[r-3,r]$ and
$$
\lim_{n\to\infty} \frac{P(S_n=x)}{P(S_n=x')}\neq 1.
$$
Moreover, if $r\leq n$ and $n^{3/4}=o(r)$, there exist points $x, x'\in\Z^2$ with $\|x\|_2, \|x'\|_2\in[r-3,r]$ and
\begin{equation}\label{divergence2}
\lim_{n\to\infty} \frac{P(S_n=x)}{P(S_n=x')}=0.
\end{equation}

\end{cor}

\begin{remark}
While it is not a priori obvious that it should happen around the circle of radius $n^{3/4}$, the loss of rotational symmetry of symmetric simple random walk is not surprising, since if $\|S_n\|_2=n$, there are only four possible locations for $S_n$.
\end{remark}





\section{Proof of Theorem \ref{LCLTthm1d}}\label{1d}
The proof of Theorem \ref{LCLTthm1d} is purely combinatorial and uses nothing more than Stirling's formula with error estimates and Taylor's theorem. We begin by stating the four lemmas needed in our proof. We will use the following version of Stirling's formula (see \cite{stirling} for a proof):
\begin{lem}\label{stirlinglemma} 
For all $n\in\N$, 
\begin{equation*}
n! = \sqrt{2\pi n}\left(n/e\right)^n e^{r_n}
\end{equation*}
with $1/(12n+1) < r_n < 1/12n$.
\end{lem}
In \cite{mortici}, a sharp bound, valid for all $n$, on the term $e^{r_n}$ of Lemma \ref{stirlinglemma} is given:
\begin{lem}\label{stirlingmortici} 
For all $n\in\N$, 
\begin{equation*}
\omega e^{1/12n} \leq e^{r_n}<  e^{1/12n},
\end{equation*}
where $r_n$ is as in Lemma \ref{stirlinglemma} and $\omega=\frac{e^{11/12}}{\sqrt{2\pi}}\approx 0.9977$.
\end{lem}

The following Lemma will prove to be useful when trying to obtain precise estimates for the factorial terms appearing in the expression for $P(S_n=x)$. 
\begin{lem}\label{errorlemma}
 If 
$r_n$ is as in Lemma \ref{stirlinglemma}, 
 then for all $n\in\N, x\in\Z$ with $|x| < n$ such that $x+n$ is even, 
\begin{eqnarray}\notag
\beta_x(n) &:= & \frac{\exp(r_n)}{\exp(r_{\frac{n+x}{2}})\exp(r_{\frac{n-x}{2}})}\\
& = &
\left\{
\begin{array}{ll} \notag
1-\frac{1}{4n}+\bigo{\frac{x^2}{n^3}+\frac{1}{n^2}}, & |x|=\ell_1(n)\\
1-\frac{1}{n}\left(\frac{3+a^2}{12(1-a^2)}\right)+\bigo{\frac{\ell_2^{a}(n)}{(1-a^2)^2n^2}}, & |x|=an+\ell_2^{a}(n), 0<a<1\\  \notag
1-\frac{1}{6\ell_3(n)}+\bigo{\frac{1}{\ell_3^2(n)}}+\bigo{\frac{1}{n}}, & |x|=n-\ell_3(n)\\
\end{array}
\right.,
\end{eqnarray}
where $\ell_1, \ell_2^{a}, \ell_3$ are as in the statement of Theorem \ref{LCLTthm1d}. Moreover, for any even $c\in\N$ such that $|c|<n$, if $|x|=n-c$, 
\begin{equation}\label{case4}
\omega e^{-1/6c}\left(1+\bigo{\frac{1}{n}}\right) \leq \beta_{x}(n)\leq \frac{1}{\omega^2} e^{-1/6c}\left(1+\bigo{\frac{1}{n}}\right),
\end{equation}
where $\omega=\frac{e^{11/12}}{\sqrt{2\pi}}$.
\end{lem}
\begin{proof} 
Note that $e^{r_n}=e^{1/12n}\left(1+\bigo{n^{-2}}\right)$ and that 
$$\beta_x(n)=\exp\left\{-\frac{3n^2+x^2}{12n(n^2-x^2)}\right\}\left(1+\bigo{n^{-2}}+\bigo{(n+x)^{-2}}+\bigo{(n-x)^{-2}}\right).$$
The first three cases follow directly from Taylor-expansion and \eqref{case4} is an easy consequence of Lemma \ref{stirlingmortici}.
\end{proof}

The last lemma we need in the proof of Theorem \ref{LCLTthm1d} is an easy application of Taylor's theorem. Some details of this straightforward calculation can be found in \cite{benesnotes}.
\begin{lem}\label{taylor}For all $z\in\R$ with $|z|<1$,
\begin{equation*}
\log(1+z)+\log(1-z)+z\left(\log(1+z)-\log(1-z)\right) = \sum_{\ell\geq 1}\frac{z^{2\ell}}{\ell(2\ell-1)}.
\end{equation*}
Moreover, 
\begin{equation}\label{Taylor} 
\sum_{\ell\geq 1}\frac{1}{2\ell(2\ell-1)}=\log 2.
\end{equation}
\end{lem}

\noindent {\it Proof of Theorem \ref{LCLTthm1d}.}
Assume $x+n$ is even. $P(S_n=\pm n) = \left(\frac{1}{2}\right)^{n}$, so Lemma \ref{taylor} yields the last line of \eqref{1dimdetails}. For $|x|<n$, Lemma \ref{stirlinglemma} 
gives
\begin{eqnarray}\label{LCLTprob}
P(S_n=x) & = & {n \choose
  \frac{n+x}{2}}\left(\frac{1}{2}\right)^{n} = \left(\frac{1}{2}\right)^{n} \frac{n!}{(\frac{n+x}{2})!(\frac{n-x}{2})!}\nonumber \\
& = & \frac{\sqrt{2\pi
    n}}{\sqrt{(n+x)\pi}\sqrt{(n-x)\pi}}\frac{(n/e)^{n}}{(\frac{n+x}{2e})^{(n+x)/2}(\frac{n-x}{2e})^{(n-x)/2}}\left(\frac{1}{2}\right)^{n}\beta_{x}(n) \\
& = & \frac{\sqrt{2\pi
    n}}{\sqrt{(n+x)\pi}\sqrt{(n-x)\pi}} \exp(\phi(n,x))\beta_{x}(n),\nonumber
\end{eqnarray}
where $\phi(n,x) = \log\left(\frac{n^{n}}{(n+x)^{(n+x)/2}(n-x)^{(n-x)/2}}\right)$ and $\beta_x(n)$ is as defined in Lemma \ref{errorlemma}. Lemma \ref{taylor} gives, for $|x| < n$,
\begin{eqnarray}\label{phi(n,x)} \notag
\phi(n,x) & = & n\log n - \frac{n+x}{2}\log(n+x) - \frac{n-x}{2}\log(n-x)\nonumber\\ 
&= & -\frac{n}{2}\left(\log(1+\frac{x}{n})+\log(1-\frac{x}{n}) +\frac{x}{n}\left(\log(1+\frac{x}{n})-\log(1-\frac{x}{n})\right)\right)\\ 
& = &-\S{\ell=1}{\infty}{\frac{1}{2\ell(2\ell-1)}\frac{x^{2\ell}}{n^{2\ell-1}}}.\nonumber
\end{eqnarray}

Also, one can easily see that if $0<a<1$ and $\ell_1, \ell_2^{a}, \ell_3$ are as in the statement of Theorem \ref{LCLTthm1d} and $c\in\N$ is constant, then
\begin{eqnarray}\label{part2} \notag
\frac{\sqrt{2\pi n}}{\sqrt{(n+x)\pi}\sqrt{(n-x)\pi}}  & = & 
\! \sqrt{\frac{2}{\pi n}}\sqrt{\frac{1}{1-x^2/n^2}}\\
 & = &
 \!\left\{
\begin{array}{ll} 
\sqrt{\frac{2}{\pi n}}\left(1+\bigo{\frac{x^2}{n^2}}\right), & |x|=\ell_1(n)\\
\sqrt{\frac{2}{\pi n}}\frac{1}{\sqrt{1-a^2}}\left(1+\bigo{\frac{\ell_2^a(n)}{(1-a^2)n}}\right), & |x|= an+\ell_2^{a}(n),\\
\frac{1}{\sqrt{\pi \ell_3(n)}}(1+\bigo{\ell_3(n)/n}), & |x| = n-\ell_3(n), \\
\frac{1}{\sqrt{\pi c}}(1+\bigo{c/n}), & |x| = n-c.
\end{array}
\right.
\end{eqnarray}





Using Lemma \ref{errorlemma}, \eqref{phi(n,x)}, and \eqref{part2} to rewrite \eqref{LCLTprob} and noting that \eqref{Taylor} yields $\exp(\phi(n,n)) = \left(\frac{1}{2}\right)^{n}$ concludes the proof of \eqref{1dimprob}. Equation \eqref{1das} is obtained by noting that it is a particular case of the first line of \eqref{1dimdetails} and that the additional error term comes from the missing terms of \eqref{phinx} in the truncated sum of \eqref{1das}.

\hfill $\square$



\section{Proofs of Theorem  \ref{LCLTthm2d} and Corollary \ref{coro}}\label{2d}

In this section, $S$ will denote planar symmetric simple random walk. In order to prove Theorem \ref{LCLTthm2d},
we will need the following standard estimate which follows from the Proposition in Section 14.8 of \cite{williams}:

\begin{figure}
\label{diagonal}
\centering%
\includegraphics[scale=0.45]{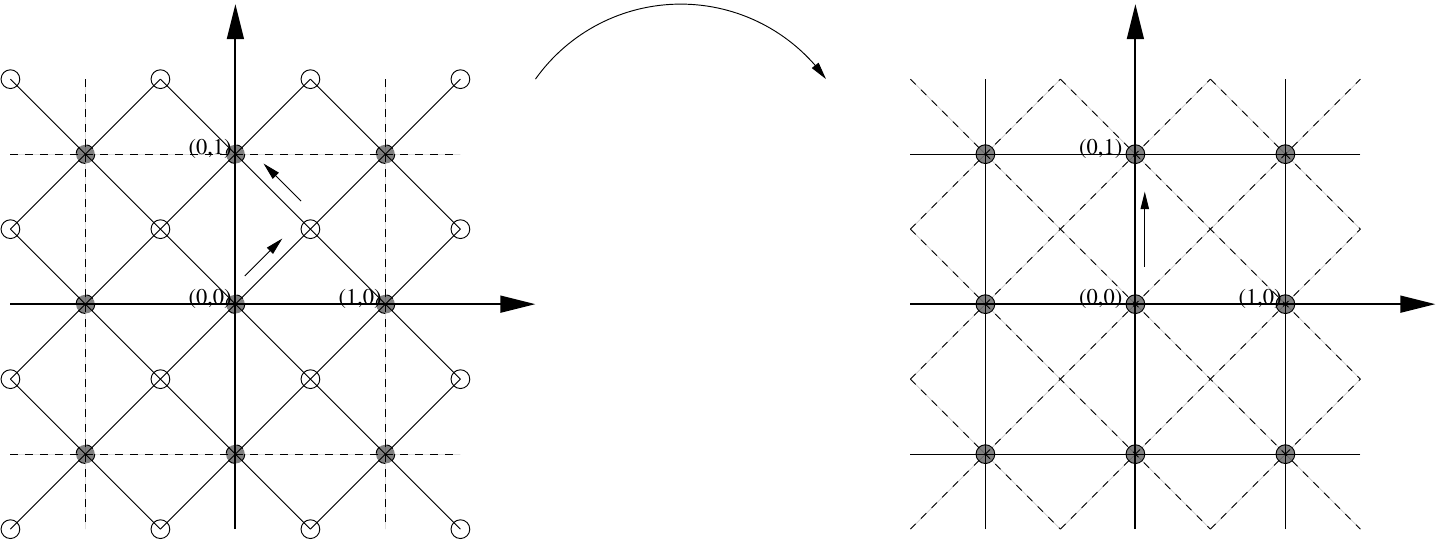}
\caption{
Two independent one-dimensional symmetric simple random walks on diagonal lattices give a symmetric simple random walk in $\Z^2$.}
\end{figure}

\begin{lem}\label{wil}
If $\phi$ is the standard normal density function and $c>0$, 
$$\int_c^{\infty} \phi(x)\,dx = \frac{1}{c}\phi(c)\left(1+\bigo{\frac{1}{c^2}}\right).$$
\end{lem}

\noindent {\it Proof of Theorem \ref{LCLTthm2d}.} 
Let $X_k^{\pi/4}$ and $X_k^{3\pi/4}$ be independent random vectors with distribution
$$P(X_k^{\pi/4}=\pm \frac{1}{\sqrt{2}}e^{i\pi/4}) = P(X_k^{3\pi/4}=\pm\frac{1}{\sqrt{2}}e^{i3\pi/4}) = \frac{1}{2}.$$
Then if we let for $n\geq 1, S^{\pi/4}(n)=\S{k=1}{n}{X_k^{\pi/4}}$ and $S^{3\pi/4}(n)=\S{k=1}{n}{X_k^{3\pi/4}}$ then $S^{\pi/4}(n)$ and $S^{3\pi/4}(n)$  are
independent symmetric simple random walks on
$\frac{1}{\sqrt{2}}e^{i\pi/4}\cdot\Z =
\{\frac{l}{\sqrt{2}}e^{i\pi/4}:l\in\Z\}$ and
$\frac{1}{\sqrt{2}}e^{i3\pi/4}\cdot\Z$, respectively, and $S_n =
(S^{\pi/4}(n),S^{3\pi/4}(n))$ is a symmetric simple random walk in $\Z^2$. 
We can use this well-known simple idea, which works only in dimension 2, and Theorem \ref{LCLTthm1d} to obtain \eqref{general} for $(x,y)\in\Z^2$ with $\|(x,y)\|_1\leq n, x+y+n$ even, from the equality
\begin{equation}\label{2dfrom1d}
P(S_n=(x,y))  =  P\left(S^{\pi/4}_n=\frac{x+y}{\sqrt{2}}\right)P\left(S^{3\pi/4}_n=\frac{y-x}{\sqrt{2}}\right).
\end{equation}
The sixteen cases of \eqref{2ddetails} are then obtained by considering 
the different regimes of \eqref{1dimprob} in each of the two probabilities of \eqref{2dfrom1d}, more specifically, by considering the different ways in which $x$ and $y$ can lead $x+y$ and $y-x$ to fall in each of the different regimes of \eqref{1dimprob}.

We now use a different approach to estimate $P(S_n=(x,y))$ when $(x,y)\in\Z^2, x+y\equiv n  (\text{mod }2)$, and $x^2+y^2\leq n^2/\log^4 n$. We will assume without loss of generality that $0\leq x\leq y\leq n/\log^2 n$.  For now we don't assume $|y^2-x^2|\leq n^{3/2}/\log n$ as the first steps of our derivation are more general.
We let $N_1$ be the 
number of steps taken by $S[0,n]$ in the horizontal direction. Then given the event $\{N_1=k\}$, $S_n$ has the distribution of $(S^{(1)}(k),S^{(2)}(n-k))$, where $S^{(1)}$ and $S^{(2)}$ are independent one-dimensional symmetric simple random walks,
 so for $(x,y)\in\Z^2$, 
\begin{equation}\label{planarprob}
P(S_n  = (x,y)) = \sum_{\substack{j=-\lfloor n/2\rfloor+x\\j+\lfloor\frac{n}{2}\rfloor+x \text{ even}}}^{\lceil n/2\rceil-y}P_{n,x,y,j},
\end{equation}
where for $-\lfloor \frac{n}{2}\rfloor\leq j\leq \lceil\frac{n}{2}\rceil$, 
\begin{equation*}
P_{n,x,y,j} = P(N_1=\lfloor \frac{n}{2}\rfloor+j)P(S^{(1)}(\lfloor\frac{n}{2}\rfloor+j)=x)P(S^{(2)}(\lceil\frac{n}{2}\rceil-j)=y).
\end{equation*}
To understand the restriction on $j$ in the sum of \eqref{planarprob}, note that by the assumption that $x+y+n$ is even, $\lfloor \frac{n}{2}\rfloor+j+x$ and $\lceil \frac{n}{2}\rceil-j+y$ are both even if and only if $j$ and $\lfloor \frac{n}{2}\rfloor+x$ have same parity.

One can show as in the proof of Theorem \ref{LCLTthm1d} that for all $j = o(n)$, 
\begin{equation}\label{est1}
P(N_1=\lfloor \frac{n}{2}\rfloor\!+\!j)  = 
\sqrt{\frac{2}{\pi n}}\exp\left\{-\sum_{\ell\geq 1}\frac{4^\ell}{2\ell(2\ell-1)}\frac{j^{2\ell}}{n^{2\ell -1}}\right\}\left(1+\bigo{\frac{1+|j|}{n}}\right).
\end{equation}
and that if $j\neq o(n)$, there is a constant $C$ such that
\begin{equation}\label{est1b}
P(N_1=\lfloor \frac{n}{2}\rfloor\!+\!j)  \leq
C \exp\left\{-\sum_{\ell\geq 1}\frac{4^\ell}{2\ell(2\ell-1)}\frac{j^{2\ell}}{n^{2\ell -1}}\right\},
\end{equation}
and use directly Theorem \ref{LCLTthm1d} and the binomial expansion of $(1+x)^{1-2\ell}$ to see that if $\lfloor \frac{n}{2}\rfloor +j + x$ is even, $\lfloor \frac{n}{2}\rfloor +j \geq x$, and $j=o(n)$, 
\begin{align}\label{est2} \notag
P(S^{(1)}(\lfloor \frac{n}{2}\rfloor &+j)=x) \\ \notag
& =   \exp\left\{-\sum_{\ell \geq 1} \frac{1}{2\ell(2\ell -1)}\frac{x^{2\ell}}{\left(\frac{n}{2}(1+\frac{2j}{n}-\frac{1}{n}\mathbbm{1}_{\{n\text{ odd}\}})\right)^{2\ell -1}}\right\} \\ \notag
&  \hspace{3pc} \cdot \sqrt{\frac{2}{\pi\frac{n}{2}(1+\frac{2j}{n}-\frac{1}{n}\mathbbm{1}_{\{n\text{ odd}\}})}}\left(1+\bigo{\frac{1}{n}}+\bigo{\frac{x^2}{n^2}}\right) \\ 
& =  \frac{2}{\sqrt{\pi n}} \exp\left\{-\sum_{k\geq 0} \sum_{\ell \geq 1} \frac{1}{4\ell(2\ell\! -\!1)}{1-2\ell \choose k}\frac{(2x)^{2\ell}(2j)^k}{n^{2\ell +k -1}}\right\} E_{x,j,n},
\end{align}
where $E_{x,j,n} = \left(1\!+\!\bigo{\frac{|j|+1}{n}}+\bigo{\frac{x^2}{n^2}}
\right)$, and if $\lceil \frac{n}{2}\rceil -j + y$ is even, $\lceil \frac{n}{2}\rceil -j \geq y$, and $j=o(n)$,
\begin{align}\label{est3} \notag
P(S^{(2)}(\lceil \frac{n}{2}\rceil -&j)\!=\!y)\\
& =  \frac{2}{\sqrt{\pi n}} \exp\left\{-\sum_{k\geq 0} \sum_{\ell \geq 1} \frac{1}{4\ell(2\ell\! -\!1)}{1-2\ell \choose k}\frac{(2y)^{2\ell}(-2j)^k}{n^{2\ell +k -1}}\right\}E_{y,j,n}.
\end{align}
Moreover, there exists a constant $C$ such that for all $n, x, y, j$, 
\begin{equation}\label{est2b}
P(S^{(1)}(\lfloor \frac{n}{2}\rfloor +j)=x)\leq C \exp\left\{-\sum_{k\geq 0} \sum_{\ell \geq 1} \frac{1}{4\ell(2\ell\! -\!1)}{1-2\ell \choose k}\frac{(2x)^{2\ell}(2j)^k}{n^{2\ell +k -1}}\right\}
\end{equation}
and
\begin{equation}\label{est3b}
P(S^{(2)}(\lceil \frac{n}{2}\rceil -j)\!=\!y)\leq C\exp\left\{-\sum_{k\geq 0} \sum_{\ell \geq 1} \frac{1}{4\ell(2\ell\! -\!1)}{1-2\ell \choose k}\frac{(2y)^{2\ell}(-2j)^k}{n^{2\ell +k -1}}\right\}.
\end{equation}
If we define for $k\geq 0$,
$$s_k=\sum_{\ell \geq 1} \frac{2^k}{4\ell(2\ell\! -\!1)}{1-2\ell \choose k}\frac{(-1)^k(2y)^{2\ell}+(2x)^{2\ell}}{n^{2\ell +k -1}},$$
combining \eqref{est1}, \eqref{est2}, and \eqref{est3} and using in the error terms our assumption that $0\leq x\leq y$ yields for $j=o(n)$ with $\lfloor \frac{n}{2}\rfloor+j+x$ even
\begin{align}\label{pnxyjestimated}\notag
P_{n,x,y,j} = \frac{4\sqrt{2}}{(\pi n)^{3/2}}\exp\bigg\{-\sum_{\ell\geq 1}\frac{4^\ell}{2\ell(2\ell-1)}&\frac{j^{2\ell}}{n^{2\ell -1}}-\sum_{k\geq 0}j^ks_k\bigg\}\\
&\cdot \left(1\!+\!\bigo{\frac{|j|+1}{n}}+\bigo{\frac{y^2}{n^2}}\right)
\end{align}
and \eqref{est1b}, \eqref{est2b}, and \eqref{est3b} imply that there is a constant $C$ such that for all $n, x, y,$ and $j$,
\begin{equation}\label{pnxyjbounded}
P_{n,x,y,j} \leq C\exp\left\{-\sum_{\ell\geq 1}\frac{4^\ell}{2\ell(2\ell-1)}\frac{j^{2\ell}}{n^{2\ell -1}}-\sum_{k\geq 0}j^ks_k\right\}.
\end{equation}

We now turn to showing \eqref{regular} and therefore assume that $y^2-x^2\leq n^{3/2}/\log n$. The main idea is to separate the sum in \eqref{planarprob} into on one hand a sum which contains the dominant terms and can be estimated precisely by comparing it to an integral and on the other hand a sum which is of smaller order of magnitude and which we can afford to bound crudely.
Note that since we are assuming that $0\leq x\leq y\leq n/\log^2(n)$, we have, for $k\geq 0$,
\begin{equation}\label{sestimate}
s_{2k} = c_{2k}\frac{x^2+y^2}{n^{2k+1}}\left(1+\bigo{\frac{y^2}{n^2}}\right) \; \text{ and } \; s_{2k+1} = c_{2k+1}\frac{y^2-x^2}{n^{2k+2}}\left(1+\bigo{\frac{y^2}{n^2}}\right),
\end{equation}
with $c_k\geq 0$ for all $k\geq 0$, so that if $|j|\leq n^{1/2}\log(n)$ and $\lfloor \frac{n}{2}\rfloor+j+x$ is even,
\begin{align*}
P_{n,x,y,j} =  \frac{4\sqrt{2}}{(\pi n)^{3/2}}  \exp&\left\{-s_0-\frac{2j^2}{n}-js_1\right\}\\
& \cdot \left(1+\bigo{\frac{|j|+1}{n}}\!+\!\bigo{\frac{j^2}{n\log^4(n)}} +\bigo{\frac{y^2}{n^2}}  \right),
\end{align*}
with all $\bigo{\cdot}$ terms actually being $o(1)$. We observe that $s_1\leq c_1(n^{1/2}\log(n))^{-1}$, which implies that there is some $C>0$ such that $ns_1^2/8 \leq C/\log^{2}(n)$. Therefore,
\begin{align*}
\exp\left\{-\frac{2j^2}{n}-js_1\right\} &=  \exp\left\{-\frac{2}{n}\left(j+\frac{ns_1}{4}\right)^2\right\}e^{ns_1^2/8}\\
& =  \exp\left\{-\frac{2}{n}\left(j+\lfloor \frac{ns_1}{4}\rfloor\right)^2\right\}e^{ns_1^2/8}\\
&\hspace{2.2pc}\cdot\left(1+\bigo{\frac{|j|+1}{n}}+\bigo{\frac{y^2-x^2}{n^2}}\right)\\
& =  \exp\left\{-\frac{2}{n}\left(j+\lfloor \frac{ns_1}{4}\rfloor\right)^2\right\}\\
&\hspace{2.2pc}\cdot\left(1+\bigo{\frac{|j|+1}{n}}+\bigo{\frac{y^2-x^2}{n^2}}+\bigo{\frac{(y^2-x^2)^2}{n^3}}\right),
\end{align*}
which implies that
\begin{align}\label{dominantsum1} \notag
\sum_{\substack{j=-\lfloor n^{1/2}\log(n)\rfloor\\j+\lfloor\frac{n}{2}\rfloor+x \text{ even}}}^{\lfloor n^{1/2}\log(n)\rfloor} P_{n,x,y,j} & = \frac{4\sqrt{2}}{(\pi n)^{3/2}} e^{-s_0} \sum_{j=-\lfloor n^{1/2}\log(n)\rfloor}^{\lfloor n^{1/2}\log(n)\rfloor}  e^{-\frac{2j^2}{n}-js_1} \\ \notag
& \hspace{3pc} \cdot \left(1\!+\!\bigo{\frac{|j|+1}{n}}\!+\!\bigo{\frac{j^2}{n\log^4(n)}} \!+\!\bigo{\frac{y^2}{n^2}}  \right)\\ \notag
& = \frac{4\sqrt{2}}{(\pi n)^{3/2}} e^{-s_0} \sum_{j=-\lfloor n^{1/2}\log(n)\rfloor+\lfloor \frac{ns_1}{4}\rfloor}^{\lfloor n^{1/2}\log(n)\rfloor+\lfloor \frac{ns_1}{4}\rfloor}  e^{-\frac{2j^2}{n}}\\ \notag
& \hspace{3pc} \cdot \left(1\!+\!\bigo{\frac{|j+\lfloor \frac{ns_1}{4}\rfloor|+1}{n}}\!+\!\bigo{\frac{(j+\lfloor \frac{ns_1}{4}\rfloor)^2}{n\log^4(n)}}
\right)\\ 
&  \hspace{6pc}\cdot \left(1\!
 +\bigo{\frac{y^2-x^2}{n^2}}+\bigo{\frac{(y^2-x^2)^2}{n^3}}\right).
\end{align}
Since $y^2-x^2\leq n^{3/2}/\log(n)$, $ns_1/4\leq Cn^{1/2}/\log(n)$, so for $n$ large enough,
\begin{equation}\label{riemann}
\sum_{\substack{j=-\lfloor n^{1/2}\log(n)/2\rfloor\\\\j+\lfloor\frac{n}{2}\rfloor+x \text{ even}}}^{\lfloor n^{1/2}\log(n)/2\rfloor}  e^{-\frac{2j^2}{n}} \leq \sum_{\substack{j=-\lfloor n^{1/2}\log(n)\rfloor+\lfloor \frac{ns_1}{4}\rfloor\\\\j+\lfloor\frac{n}{2}\rfloor+x \text{ even}}}^{\lfloor n^{1/2}\log(n)\rfloor+\lfloor \frac{ns_1}{4}\rfloor}  e^{-\frac{2j^2}{n}}  \leq  \sum_{\substack{j=-2\lfloor n^{1/2}\log(n)\rfloor\\j+\lfloor\frac{n}{2}\rfloor+x \text{ even}}}^{2\lfloor n^{1/2}\log(n)\rfloor}  e^{-\frac{2j^2}{n}}.
\end{equation}
Recognizing the outer sums of \eqref{riemann} to be Riemann sums and using the monotonicity properties of the function $e^{-2x^2}$, we can use Lemma \ref{wil} to write
\begin{align}\label{leftsum} \notag
\sum_{\substack{j=-\lfloor n^{1/2}\log(n)/2\rfloor\\j+\lfloor\frac{n}{2}\rfloor+x \text{ even}}}^{\lfloor n^{1/2}\log(n)/2\rfloor}  e^{-\frac{2j^2}{n}} & \geq 2\sum_{\substack{j=0\\j+\lfloor\frac{n}{2}\rfloor+x \text{ even}}}^{\lfloor n^{1/2}\log(n)/2\rfloor}  e^{-\frac{2j^2}{n}} - 1\\ \notag
& \geq \sqrt{n}\int_0^{\log (n)/2}e^{-2x^2}\,dx +\bigo{1}\\ \notag
& = \frac{\sqrt{n}}{2}\left(\sqrt{\frac{\pi}{2}}-\int_{\log(n)}^{\infty}e^{-x^2/2}\,dx\right)+\bigo{1}\\
& =\frac{1}{2}\sqrt{\frac{\pi n}{2}}\left(1
+\bigo{1/\sqrt{n}}\right)
\end{align} 
and similarly,
\begin{equation}\label{rightsum}
\sum_{\substack{j=-2\lfloor n^{1/2}\log(n)\rfloor\\j+\lfloor\frac{n}{2}\rfloor+x \text{ even}}}^{2\lfloor n^{1/2}\log(n)\rfloor}  e^{-\frac{2j^2}{n}} \leq \sqrt{n}\int_{-4\log(n)}^{4\log(n)}e^{-2x^2}\,dx +\bigo{1} = \frac{1}{2}\sqrt{\frac{\pi n}{2}}\left(1
+\bigo{1/\sqrt{n}}\right).
\end{equation} 
Combining \eqref{riemann}, \eqref{leftsum}, and \eqref{rightsum}, we get
\begin{equation}\label{centralsum}
\sum_{\substack{j=-\lfloor n^{1/2}\log(n)\rfloor+\lfloor \frac{ns_1}{4}\rfloor\\j+\lfloor\frac{n}{2}\rfloor+x \text{ even}}}^{\lfloor n^{1/2}\log(n)\rfloor+\lfloor \frac{ns_1}{4}\rfloor}  e^{-\frac{2j^2}{n}} =\frac{1}{2}\sqrt{\frac{\pi n}{2}}\left(1+\bigo{1/\sqrt{n}}\right).
\end{equation}
Therefore,  \eqref{dominantsum1} and \eqref{centralsum} give the estimate
\begin{equation}\label{dominantsum2}
\sum_{\substack{j=-\lfloor n^{1/2}\log(n)\rfloor\\j+\lfloor\frac{n}{2}\rfloor+x \text{ even}}}^{\lfloor n^{1/2}\log(n)\rfloor} P_{n,x,y,j}  = \frac{2}{\pi n} e^{-s_0} \left(1\!
 +\bigo{\frac{y^2-x^2}{n^2}}+\bigo{\frac{(y^2-x^2)^2}{n^3}}\right).
\end{equation}
We will now show that $\sum_{|j| \geq \lfloor n^{1/2}\log(n)\rfloor} P_{n,x,y,j}$ is of smaller order than the sum obtained in \eqref{dominantsum2}. Differentiating with respect to $j$ the exponent in the expression for $P_{n,x,y,j}$ in \eqref{pnxyjbounded},
\begin{equation}\label{fnxyj}
f(j) = f_{n,x,y}(j) = -s_0 -\sum_{\ell\geq 1}\frac{4^\ell}{2\ell(2\ell-1)}\frac{j^{2\ell}}{n^{2\ell -1}}-\sum_{k\geq 1}j^ks_k,
\end{equation}
shows that $f$ has a unique maximum which occurs at $j_0=-\lfloor \frac{y^2-x^2}{2n}\rfloor$ or $j'_0=-\lceil \frac{y^2-x^2}{2n}\rceil$.

Note that when $y^2-x^2\leq n^{3/2}/\log(n), j_0$ and $j'_0$ belong to the index set of the sum in \eqref{dominantsum2}. It is easy to see that in that case, on the set $\{j\in\Z: |j|\geq \lfloor n^{1/2}\log(n)\rfloor\}$,
$f$ is maximal when $j=-\lfloor n^{1/2}\log(n)\rfloor$ and, using \eqref{sestimate}, that
$$f(-\lfloor n^{1/2}\log(n)\rfloor)=-s_0-2\log^2(n)+\bigo{1}.$$
Therefore, using \eqref{pnxyjbounded}, we see that there exists a constant $C>0$ such that
\begin{equation}\label{partialsum2}
\sum
P_{n,x,y,j} \leq nP_{n,x,y,-\lfloor n^{1/2}\log(n)\rfloor} \leq Cn^{-1/2}e^{-s_0-2\log^2(n)},
\end{equation}
where the sum is over the set 
$$\{j\in\Z: -\lfloor n/2\rfloor +x\leq j\leq -\lfloor n^{1/2}\log n\rfloor \text{ or } \lfloor n^{1/2}\log n\rfloor\leq j \leq \lfloor n/2\rfloor -y\}.$$
Combining \eqref{planarprob}, \eqref{dominantsum2}, and \eqref{partialsum2} now yields \eqref{regular}.

\hfill $\square$

\noindent {\it Proof of Corollary \ref{coro}} 
Suppose $s-3\leq \|(x,y)\|_2\leq s=o(n^{3/4})$. Then $x, y=o(n^{3/4})$, implying that $x+y, y-x=o(n^{3/4})$, so by \eqref{general} and \eqref{2ddetails}, if $x+y+n$ is even,
\begin{eqnarray*}
P(S_n=(x,y)) & = & \frac{2}{\pi n}\exp\left\{-\frac{1}{2}\frac{(x+y)^2+(y-x)^2}{n}\right\}(1+o(1))\\
& = & \frac{2}{\pi n}\exp\left\{-\frac{x^2+y^2}{n}\right\}(1+o(1))=  \frac{2}{\pi n}\exp\left\{-\frac{s^2}{n}\right\}(1+o(1)),
\end{eqnarray*}
which proves approximate rotational symmetry in the disk of radius $r$ for any $r=o(n^{3/4})$.

Note that \eqref{divergence2} is obvious for $n/\sqrt{2}+3<r\leq n$ since for such $r$, we can find two points at distance from the origin between $r-3$ and $r$ that have $L^1$ norms with same parity, exactly one of which is attainable by the random walk, so we focus on $r\leq n/\sqrt{2}+3$
. For any such $r$, there exist $s\in [r-3,r]$ with $(\sqrt{s/2}, \sqrt{s/2})\in\Z^2$ and $s'\in [r-3,r]$ with $(s', 0)\in\Z^2$ such that the $L^1$ norms of the two points have same parity as $n$. By \eqref{general}
$$\log P(S(n)=(s',0))\sim -2\sum_{\ell\geq 1}\frac{1}{2\ell(2\ell-1)}\frac{s'^{2\ell}}{n^{2\ell-1}}$$
and
$$\log P(S(n)=(\sqrt{s/2},\sqrt{s/2}))\sim -\sum_{\ell\geq 1}\frac{1}{2\ell(2\ell-1)}\frac{2^\ell s^{2\ell}}{n^{2\ell-1}}.$$
Since for 
$r\geq n^{3/4}$, 
the exponent of $P(S(n)=(\sqrt{s/2},\sqrt{s/2}))$ is smaller than that of $P(S(n)=(s',0))$, we see that 
$$\lim_{n\to\infty}\frac{P(S(n)=(\sqrt{s/2},\sqrt{s/2}))}{P(S(n)=(s',0))}<1.$$
In particular, if $n^{3/4}=o(r)$, 
$$\lim_{n\to\infty}\frac{P(S(n)=(\sqrt{s/2},\sqrt{s/2}))}{P(S(n)=(s',0))}=0.$$
This concludes the proof.
\hfill $\square$



\section*{Appendix}

In this appendix, we give a brief summary of some of the work done throughout the 20th century that led towards the sharpest currently available general local limit theorems that are relevant to the problem discussed in this paper. We also verify that the expression obtained in \eqref{general} is indeed the simple random walk large deviations rate function 

In 1929, Khinchine (\cite{khinchine}) obtained an asymptotic expression for the quantity in \eqref{mainprob} for one-dimensional simple random walk, both symmetric and not, for all $|x|=o(n)$ which, in the symmetric case, corresponds to the first line in the expression of $P(S_n=x)$ in Theorem \ref{LCLTthm1d} of this paper. His approach, as ours, was combinatorial.

In 1938, Cram\'{e}r studied the ratio
$$\frac{1-P\left(S_n\leq \sigma\sqrt{n}x\right)}{1-\Phi(x)} \;\text{ as } x\to\infty,$$
where $\Phi$ is the standard normal distribution function, $X_1, X_2, \ldots $ are independent, identically distributed random variables for which the moment generating function is defined in a neighborhood of 0, $E[X_1]=0, \rm{Var}(X_1)=\sigma^2$, and $S_n=\sum_{i=1}^nX_i$. Making use of the Esscher transform (see \cite{esscher}; some authors also call this transform the Cram\'{e}r transform), he showed 
that for $x=o(\sqrt{n}/\log n)$,
\begin{equation}\label{cramer}
\frac{1-P\left(S_n\leq \sigma\sqrt{n}x\right)}{1-\Phi(x)}=\exp\left\{\frac{x^3}{\sqrt{n}}\lambda\left(\frac{x}{\sqrt{n}}\right)\right\}\left(1+\bigo{\frac{x\log n}{\sqrt{n}}}\right),
\end{equation}
where $\lambda(z)$ is power series that converges for sufficiently small $|z|$, the coefficients of which can be expressed in terms of the cumulants (also called semi-invariants) of $X_1$. Obtaining an explicit expression for $\lambda$ presented in this form is difficult even for the simplest choice of $X_1$, as the coefficients of the series become increasingly complicated with each term. However, it turns out that there is a more convenient way of stating Cram\'{e}r's theorem in terms of the Legendre-Fenchel transform of the logarithmic moment generating function. See below in this section.

Feller (\cite{feller}) and Petrov (\cite{petrov}) extended Cram\'{e}r's argument and result to random variables that are not identically distributed. They also extended the range of points for which their expansions are valid to all $x=o(\sqrt{n})$. Cram\'{e}r's result is a consequence of Petrov's but not of Feller's whose assumptions are more restrictive that Cram\'{e}r's in the identically distributed case.

The results just mentioned give asymptotics of the tail probabilities of the random variable $\sum_{i=1}^nX_i/(\sigma\sqrt{n})$. In \cite{richter1}, Richter extended the ideas developed by Cram\'{e}r and used an inversion formula and the saddlepoint method to show that if independent lattice random variables $X_i$ (that is, whose possible values are all of the form $a+kh, k\in\Z$ for some $a$ and $h$, which is taken to be maximal) with mean 0, variance $\sigma^2$, and finite moment generating function in some neighborhood of 0, take values in $\Z$ 
then, with $x=\frac{an+kh}{\sigma\sqrt{n}}$, for $x>1, x=o(\sqrt{n})$, 
\begin{equation}\label{richter1}
\frac{P(S_n=\sigma\sqrt{n}x)}{\phi(x)}=\frac{h}{\sigma\sqrt{n}}\exp\left\{\frac{x^3}{\sqrt{n}}\lambda\left(\frac{x}{\sqrt{n}}\right)\right\}\left(1+\bigo{\frac{x}{\sqrt{n}}}\right),
\end{equation}
where $\lambda$ is the same as before and $\phi$ is the standard normal density function.

In \cite{richter2}, Richter applied the same ideas as in \cite{richter1} to irreducible, aperiodic $d$-dimensional random walks with finite moment generating function in a neighborhood of 0 to obtain a multi-dimensional extension of \eqref{richter1} in which an infinite sum of multilinear forms depending on the cumulants of the underlying random variables plays the role of $\lambda$ in \eqref{richter1}. 




It turns out that \eqref{cramer} has a much more convenient expression. Indeed, Cram\'{e}r's method was to express the ratio in  \eqref{cramer} in terms of the logarithmic moment generating function and quantities related to the Esscher transform of the summed random variables. The Esscher transform is a one-parameter family of random variables and a suitable choice of the parameter as a function of the point $x$ considered was at the heart of Cram\'{e}r's method. It follows directly from the proof of \eqref{cramer} that with $M$ the moment generating function of $X_1$,
$$P\left(S_n\geq \sigma\sqrt{n}x\right)  \approx  M(h)^ne^{-h\frac{\sigma x}{\sqrt{n}}n} =  \exp\left\{n\left(\log M(h)-h\frac{\sigma x}{\sqrt{n}}\right)\right\}.$$
Letting $h=\frac{x}{\sigma\sqrt{n}}$ leads, via Taylor expansions of the logarithmic moment generating function, to \eqref{cramer}, but also to the following expression where 
\begin{equation}\label{leg-fench}
\Lambda(\alpha)=\sup_{\lambda\in\R}\{\lambda \alpha-\log M(\lambda)\}
\end{equation}
is the Legendre-Fenchel transform of the logarithmic moment generating function:
$$P\left(S_n\geq \sigma\sqrt{n}x\right) \approx \exp\{-n\Lambda(\sigma x/\sqrt{n})\}.$$
Indeed, one can verify that the supremum in \eqref{leg-fench} is attained at a point which, for $x=o(\sqrt{n})$, is asymptotic to $\lambda=\frac{x}{\sigma\sqrt{n}}$. This observation was of use in \cite{borovkovmogulskii1}, where Borovkov and Mogul'skij extended Richter's results to points within $\bigo{n}$ of the mean and expressed the exponent in terms of $\Lambda$. In particular, for simple random walk, their result is that for $x\in\Z^d$ with $\|x\|_1+n$ even,
\begin{equation}\label{boromog}
P(S_n=x)=\frac{e^{-n\Lambda(\alpha)}}{(2\pi n)^{d/2}\sigma(\alpha)}(1+\eps(n,\alpha)),
\end{equation}
where $\alpha=x/n$,
\begin{equation}\label{legendrefenchel}
\Lambda(\alpha)=\sup_{\lambda\in\R^d}\{\langle \lambda, \alpha\rangle-\log M(\lambda)\}, \, \alpha\in\mathcal{X}
\end{equation}
is the Legendre-Fenchel transform of the moment-generating function $M$, $\sigma(\alpha) = \sqrt{\det(\Lambda''(\alpha))^{-1}}$, and for any open subset $G$ of $\mathcal{X}=\{x\in\R^d:\|x\|_1\leq 1\}$, 
$$\lim_{n\to\infty}\sup_{x\in\Z^d, \alpha\in G}|\eps(n,\alpha)|=0.$$

We conclude this appendix by verifying that in the case of planar symmetric simple random walk, the exponent in \eqref{boromog} coincides with the exponent we obtain in Theorem \ref{LCLTthm2d}. As the present paper shows, this exponent is valid for all points that are attainable by the walk.

It is easy to verify that with the definition \eqref{legendrefenchel} the expression in \eqref{phinx} is, as expected, equal to $-n\Lambda(x/n)$ when $d=1$, so we omit this calculation here. When $d=2$, the exponent in \eqref{general} is also equal to $-n\Lambda(x/n,y/n)$, though this takes considerably more work to verify. We outline here the main steps of that derivation.

For $\alpha=(\alpha_1,\alpha_2)\in \{x\in\R^2:\|x\|_1\leq 1\}$,
$$\Lambda(\alpha)=\sup_{(\lambda_1,\lambda_2)\in\R^2}\{\lambda_1\alpha_1+\lambda_2\alpha_2-\log (\frac{1}{2}(\cosh \lambda_1+\cosh \lambda_2))\}.$$
The supremum is attained when $\alpha_i=\frac{\sinh \lambda_i}{\cosh \lambda_1 + \cosh \lambda_2}$, i=1, 2, which is the case for real $\lambda_1, \lambda_2$ if and only if
\begin{align*}
\lambda_1&=\log \sqrt{\frac{(\alpha_1+1)^2-\alpha_2^2}{(\alpha_1-1)^2-\alpha_2^2}},\\
\lambda_2&=\log \sqrt{\frac{(\alpha_2+1)^2-\alpha_1^2}{(\alpha_2-1)^2-\alpha_1^2}}.
\end{align*}
With these values of $\lambda_1, \lambda_2$, one can check that
\begin{eqnarray*}
\Lambda(\alpha)&=&\frac{1}{2}\bigg((1+\alpha_1+\alpha_2)\log(1+\alpha_1+\alpha_2)+(1+\alpha_1-\alpha_2)\log(1+\alpha_1-\alpha_2)\\
&&\hspace{1pc}+(1-\alpha_1+\alpha_2)\log(1-\alpha_1+\alpha_2)+(1-\alpha_1-\alpha_2)\log(1-\alpha_1-\alpha_2)\bigg).
\end{eqnarray*}
Using the Taylor expansion of $\log(1+x)$ in this last expression, gives, after simplification, 
$$\Lambda(\alpha)=\sum_{\ell \geq 1}\frac{1}{2\ell(2\ell-1)}\left((\alpha_1+\alpha_2)^{2\ell}+(\alpha_1-\alpha_2)^{2\ell}\right),$$
so that indeed, $-n\Lambda(x/n,y/n)$ is the exponent in \eqref{general}.




\section*{Acknowledgements}
The author 
gratefully acknowledges support for his research through PSC-CUNY Awards \#60130-3839, \#62408-0040, and \# 63405-0041.

\bibliographystyle{amsplain}
\bibliography{LCLT}
\end{document}